\documentclass{amsart}
\usepackage{amsfonts,amssymb,amscd,amsmath,mathabx,enumerate,verbatim,calc,mathrsfs}
\usepackage{amscd,amssymb,amsopn,amsmath,amsthm,graphics,amsfonts,enumerate,verbatim,calc}
\usepackage[dvips]{graphicx}
\usepackage[colorlinks=true,linkcolor=red,citecolor=blue]{hyperref}
\usepackage{mathrsfs}
\usepackage[T1]{fontenc}
\usepackage[autostyle]{csquotes}
\usepackage[colorlinks=true,linkcolor=red,citecolor=blue]{hyperref}
\usepackage{tikz-cd}
\usetikzlibrary{matrix,arrows,backgrounds}
\usepackage{tikz}
\usetikzlibrary{matrix,arrows,decorations.pathmorphing}

\input xy
\xyoption{all}
\calclayout

\newcommand*{\medcap}{\mathbin{\scalebox{1.5}{\ensuremath{\cap}}}}%
\newcommand{\rt}{\rightarrow}

\newcommand{\st}{\stackrel}

\newcommand{\fm}{\frak{m}}

\newcommand{\MCM}{{\rm{CM}}}
\newcommand{\Tr}{{\rm{Tr}}}
\newcommand{\D}{{\rm{D}}}

\newcommand{\rep}{{\rm{rep}}}

\newcommand{\Prod}{{\rm{Prod}}}

\newcommand{\Rep}{{\rm{Rep}}}

\newcommand{\im}{{\rm{Im}}}
\newcommand{\op}{{\rm{op}}}

\newcommand{\add}{{\rm{add}}}

\newcommand{\Add}{{\rm{Add}}}

\newcommand{\Cosyz}[2]{\mho_{#1}(#2)}

\newcommand{\gen}{{\rm{gen}}}
\newcommand{\cogen}{{\rm{cogen}}}
\newcommand{\codim}{{\rm{codim}}}

\newcommand{\pd}{{\rm{pd}}}
\newcommand{\id}{{\rm{id}}}
\newcommand{\rid}{{\rm{rid}}}
\newcommand{\fd}{{\rm{fd}}}

\newcommand{\Coker}{{\rm{Coker}}}
\newcommand{\Ker}{{\rm{Ker}}}

\newcommand{\Ann}{{\rm{Ann}}}

\newcommand{\rfmod}[1]{\mbox{\rm{mod}--}{#1}}
\newcommand{\lfmod}[1]{{#1}\mbox{--\rm{mod}}}
\newcommand{\Tor}{{\rm{Tor}}}
\newcommand{\Hom}{{\rm{Hom}}}
\newcommand{\Ext}{{\rm{Ext}}}
\newcommand{\End}{{\rm{End}}}

\newcommand{\Hop}{{\rm{op}}}

\newtheorem{theorem}{Theorem}[section]
\newtheorem{corollary}[theorem]{Corollary}
\newtheorem{lemma}[theorem]{Lemma}

\newtheorem{notation}[theorem]{Notation}

\theoremstyle{definition}
\newtheorem{definition}[theorem]{Definition}
\newtheorem{example}[theorem]{Example}

\newtheorem{remark}[theorem]{Remark}

\theoremstyle{plain}

\theoremstyle{definition}

\numberwithin{equation}{section}

\begin{document}

\title[Cotilting invariance of ...]
{Cotilting invariance of the Auslander-Reiten conjecture}

\author[Kamran Divaani-Aazar, Ali Mahin Fallah, Massoud Tousi]
{Kamran Divaani-Aazar$^{(*)}$, Ali Mahin Fallah and Massoud Tousi}

\address{K. Divaani-Aazar, Department of Mathematics, Faculty of Mathematical Sciences, Alzahra University,
Tehran, Iran.}
\email{kdivaani@ipm.ir}

\address{School of Mathematics, Institute for Research in Fundamental Sciences (IPM), P.O. Box: 19395-5746,
Tehran, Iran.}
\email{amfallah@ipm.ir, ali.mahinfallah@gmail.com}

\address{Tousi, Department of Mathematics, Faculty of Mathematical Sciences, Shahid Beheshti University,
Tehran, Iran.}
\email{mtousi@ipm.ir}


\keywords{Auslander-Reiten conjecture; co-Noetherian ring; cotilting module; Noetherian algebra;
product-complete module; self-orthogonal module; tilting module; Wakamatsu tilting module.\\
The research of the second author is supported by a grant from IPM (No.1402160015).\\
$^{(*)}$ Corresponding author}

\begin{abstract} Let $R$ be an associative ring with identity, and let $T$ be a tilting right $R$-module, with
$S=\End(T)$. It is known that if $R$ is a Noetherian algebra that satisfies the Auslander-Reiten conjecture,
then so is $S$. In this paper, we consider the dual situation where $C$ is a cotilting right $R$-module, with
$S=\End(C)$.  We investigate the invariance of the property of satisfying the Auslander-Reiten conjecture when
passing from $R$ to $S$.
\end{abstract}

\maketitle

\section{Introduction}

Nakayama's conjecture \cite{N1}, a long-standing open problem in the representation theory of Artin algebras, asserts
that an Artin algebra $R$ is self-injective if all terms in its minimal injective resolution are projective. In their
study of this conjecture, Auslander and Reiten \cite{AR} conjectured that for an Artin algebra $R$, every indecomposable
injective left $R$-module appears as a direct summand in one of the terms of the minimal injective resolution of $R$.
This conjecture is now known as the generalized Nakayama conjecture. Auslander and Reiten \cite{AR} further proved that
the generalized Nakayama conjecture holds for all Artin algebras if and only if the following conjecture is true for all
Artin algebras. Throughout, the term {\it ring} will mean an associative ring with identity.

\vspace{.3cm}
\hspace{0.3cm}  {\bf Conjecture (ARC):} Let $R$ be a ring, and $M$ a finitely generated left $R$-module. If 
$\Ext^i_R(M,M\oplus R)=0$ for all $i>0$, then $M$ is projective.
\vspace{.3cm}

This conjecture, now known as the Auslander–Reiten conjecture, is a consequence of the finitistic dimension conjecture.
The finitistic dimension conjecture states that, for an Artin algebra $R$, the supremum of the projective dimensions of 
all finitely generated left $R$-modules with finite projective dimension is finite. If the finitistic dimension conjecture 
holds for all Artin algebras, then (ARC) also holds for all Artin algebras.

Tilting theory is crucial in the representation theory of Artin algebras. Miyashita \cite{M} explored the dual concept
of cotilting modules over Noetherian rings in 1986. Auslander and Reiten further investigated cotilting modules over
Artin algebras, viewing them as a non-commutative analogue of dualizing modules over commutative Noetherian rings
\cite{ART, AR1}. For Artin algebras, these definitions coincide.

Wei proved in \cite{W1} that (ARC) is preserved under tilting equivalences of Artin algebras.
Specifically, if $T$ is a tilting right module over an Artin algebra $R$ with the endomorphism ring $S=\End(T)$, and
if $R$ satisfies (ARC), then so does $S$. The second author generalized this result to Noetherian
algebras in \cite{Ma}. Furthermore, Wei showed in \cite{W1} that (ARC) is also preserved under
cotilting equivalences of Artin algebras. That is, if $C$ is a cotilting right module over an Artin algebra $R$ with
the endomorphism ring $S=\End(C)$, and if $R$ satisfies (ARC), then $S$ does as well. This
raises the natural question of whether a similar preservation property holds for Noetherian algebras. To address this
question, we recall the following conjecture from \cite{EM}, which serves as a dual to (ARC):

\vspace{.3cm}
\hspace{0.3cm}  {\bf Conjecture (DARC):}  Let $R$ be a ring, and $M$ an Artinian left $R$-module. If 
$\Ext^i_R(M\oplus I,M)=0$ for every injective left $R$-module $I$ and all $i>0$, then $M$ is injective.
\vspace{.3cm}

This conjecture is motivated by the equivalence between (DARC) and (ARC) for Artin algebras.
An Artin algebra $R$ satisfies (ARC) if and only if $R^{\Hop}$ satisfies (DARC). Furthermore,
if a Noetherian algebra $R$ over a commutative Noetherian complete semi-local ring satisfies (ARC),
then $R^{\Hop}$ satisfies (DARC). As our main result, we prove:

\begin{theorem}\label{1.1} Let $R$ be a right Artinian ring, let $C$ be a cotilting right $R$-module, and set $S=\End(C_R)$. 
\begin{itemize}
\item[(1)] If $S$ satisfies (ARC), then $R^{\Hop}$ satisfies (DARC).
\item[(2)]  The converse holds under the additional assumptions that $R$ is right co-Noetherian and $S$ is left Artinian.
\end{itemize}
\end{theorem}

We prove the above result in Theorem \ref{3.12}. Clearly, the above result generalizes Wei’s aforementioned result; see
\cite[Main Theorem and Corollary 3.7]{W1}. 

We end the paper by presenting an example showing that there exist rings $S$ satisfying (ARC) by Theorem \ref{1.1} whenever 
the ring $R^{\Hop}$ satisfies (DARC).

\section{Preliminaries}

Throughout, $R$ is an associative ring with identity.  We use the notation $M_R$ (respectively, $_RM$) to denote a right
(respectively, left) $R$-module. The category of all (respectively, finitely generated) right $R$-modules is denoted by
$\text{Mod-}R$ (respectively, $\text{mod-}R$), while $R\text{-Mod}$ (respectively, $R\text{-mod}$) represents the category
of all (respectively, finitely generated) left $R$-modules.

Let $M_R$ be an $R$-module. We denote by $\Add(M)$ (respectively, $\add(M)$) the class of right $R$-modules that are isomorphic
to a direct summand of a direct sum of copies (respectively, finitely many copies) of $M$. Dually, $\Prod(M)$ denotes the
class of right $R$-modules that are isomorphic to a direct summand of a direct product of copies of $M$.

Let $\mathscr{C}$ be a full subcategory of the category of $\text{Mod-}R$, and $n$ a non-negative integer. We denote by
$\widehat{\mathscr{C}}_n$ the category of all $R$-modules $M_R$ such that there exists an exact sequence $$0\rt C_n \rt \cdots
\rt C_1 \rt C_0 \rt M \rt 0$$ with each $C_i\in \mathscr{C}$.  Dually, we then denote by $\widecheck{\mathscr{C}}_n$ the category
of all $R$-modules $M_R$ such that there exists an exact sequence $$0\rt M\rt C^0\rt C^1\rt \cdots\rt C^n\rt 0$$ with each
$C^i\in \mathscr{C}$.

For an $R$-module $T_R$, denote by $\gen^* (T_R)$ the class consisting of the $R$-modules $N_R$ for which there exists an exact
sequence of the form $$\cdots\st{f_2}\longrightarrow T_1 \st{f_1}\longrightarrow T_0 \st{f_0}\longrightarrow N\longrightarrow 0$$
with each $T_i\in \add(T)$ and
$\Ext^1_R(T,\Ker f_i)=0$ for all $i\geqslant0$. Dually, $\cogen^* (T_R)$ denotes the class of all $R$-modules $N_R$ for which
there exists an exact sequence of the form $$0\longrightarrow N \st{f^{-1}}\longrightarrow T^0
\st{f^0}\longrightarrow T^1 \st{f^1}\longrightarrow \cdots$$ with each $T^i\in \add(T)$ and $\Ext^1_R(\Coker f^i,T)=0$ for all
$i\geqslant-1$.

For an $R$-module $M_R$, the symbol $M^{\perp}$ denotes the  subcategory of $R$-modules $N_R$ such that $\Ext^i_R(M,N)=0$
for all $i\geq 1$. Dually, $~^{\perp}M$ denotes the subcategory of $R$-modules $N_ R$ such that $\Ext^i_R(N,M)=0$ for all
$i\geq 1$.

For an $R$-module $M_R$ and a positive integer $n$, the symbol $\Omega^n_{R}M$ denotes the $n$-th syzygy of $M$, with the
convention that $\Omega^0_{R}M=M$. We also recall the definition of cosyzygies. Given an $R$-module $M_R$, and an injective
resolution $$0\rt M\rt E^0\rt E^1\rt\cdots \rt E^{i-1} \to E^i\rt \cdots$$ for $M$, the $n$-th cosyzygy of $M$ is defined as
$\Cosyz nM=\im(E^{n-1}\to E^n)$ for every $n\geq 1$. Additionally, we set $\Cosyz 0M=M$.

The following lemma will be useful in our investigation. Its proof is an immediate consequence of the dimension-shift property
of Ext modules and is therefore omitted.

\begin{lemma}\label{2.1} Let $M_R, N_R$ and $X_R$ be $R$-modules, and $n\in \mathbb{N}$.
\begin{itemize}
\item[(i)] If $0\rt M\rt K_0\rt K_1\rt \cdots \rt K_{n-1}\rt X\rt 0$ is an exact sequence with each $K_j$ in $N^\perp$,
then $\Ext^{i+n}_R(N,M)\cong \Ext^{i}_R(N,X)$ for all $i\geq 1$.
\item[(ii)] If $0\rt X\rt L_0\rt L_1\rt  \cdots \rt L_{n-1}\rt M\rt 0$ is an exact sequence with each $L_j$ in $~^\perp N$,
then $\Ext^{i+n}_R(M,N)\cong \Ext^{i}_R(X,N)$ for all $i\geq 1$.
\end{itemize}
\end{lemma}

Recall that an $R$-module $M_R$ is called {\it self-orthogonal} if $\Ext^i_R(M,M)=0$ for all $i\geq 1$. We next recall the
definitions of tilting and Wakamatsu tilting modules.

\begin{definition}\label{2.2} Let $n\in \mathbb{N}_0$. A finitely generated self-orthogonal $R$-module $T_R$ is called
$n$-{\it tilting} if
\begin{itemize}
\item[(i)]  $\pd(T_{R})\leq n$, and
\item[(ii)] $R\in \widecheck{\add(T)}_n$.
\end{itemize}
\end{definition}

\begin{definition}\label{2.3} A self-orthogonal $R$-module $T_R$ is called  {\it Wakamatsu tilting} if
\begin{itemize}
\item[(i)]  $T_R\in \gen^*(R)$, and
\item[(ii)] $R\in \cogen^*(T_R)$.
\end{itemize}
\end{definition}

The definitions of tilting left modules and Wakamatsu tilting left modules are similar. It is straightforward to verify
that every tilting module is also a Wakamatsu tilting module. According to \cite[Corollary 3.2]{W}, Wakamatsu tilting
modules can be characterized as follows:

\begin{lemma}\label{2.4} For a bimodule $_ST_R$, the following are equivalent:
\begin{itemize}
\item[(i)] $T_R$ is a Wakamatsu tilting module with $S\cong \End(T_R)^{\Hop};$
\item[(ii)] $_ST$ is a Wakamatsu tilting module with $R\cong \End(_ST)^{\Hop};$
\item[(iii)] One has
\begin{itemize}
\item[(1)]  $T_R\in \gen^*(R)$ and $_ST\in \gen^*(S)$.
\item[(2)]  $S\cong \End(T_R)^{\Hop}$ and $R\cong \End(_ST)^{\Hop}$.
\item[(3)] The modules $T_R$ and $_ST$ are self-orthogonal.
\end{itemize}
\end{itemize}
\end{lemma}

There are several definitions of cotilting modules. In this paper, we adopt the definition provided by Miyashita in 1986;
see \cite[page 142]{M}. This notion of cotilting modules is also discussed in \cite{HT, M1, M2}.

\begin{definition}\label{2.5} Let $C_R$ be a finitely generated $R$-module and $S=\End(C_R)$. Assume that the rings $R$ and
$S$ are right Noetherian and left Noetherian; respectively. We say $C_R$ is {\it cotilting} if it is a Wakamatsu tilting
$R$-module and both $\id(C_R)$ and $\id(_SC)$ are finite.
\end{definition}

Left cotilting modules are defined analogously. Below, we present some typical examples of cotilting modules.

\begin{example}\label{2.6}
\begin{itemize}
\item[(i)] For an Artin  algebra $R$  and a tilting left $R$-module $T$, the dual of $T$ is a cotilting right $R$-module.
\item[(ii)] Let $R$ be an Iwanaga–Gorenstein ring. Then $R$, regarded as either a right or a left $R$-module, is cotilting.
\item[(iii)]  Let $\Gamma$ be a Cohen-Macaulay local ring with a dualizing module $\omega$, and $R$ be a $\Gamma$-order
(i.e. $R$ is a $\Gamma$-algebra and is maximal Cohen-Macaulay as a $\Gamma$-module). Then, by \cite[Proposition 2.12]{M2}, $\Hom_\Gamma(R,\omega)$ is a cotilting right $R$-module.
\item[(iv)] Let $C_R$ be a cotilting $R$-module, and let $S=\End(C_R)$. Then, the rings
$\widetilde{R}=\left(\begin{matrix}
{R}& {R}\\
{0}&{R}
\end{matrix}
\right)$  and
$\widetilde{S}=\left(\begin{matrix}
{S}& {S}\\
{0}&{S}
\end{matrix}
\right)$
are right Noetherian and left Noetherian; respectively.  Set $\widetilde{C}=(0\rt C)\oplus (C\rt C)$. By applying the
arguments from Lemmas 3.4 and 3.7 in \cite{Zh} and using Lemma \ref{2.4}, we can conclude that $\widetilde{C}_{\widetilde{R}}$
is a cotilting $\widetilde{R}$-module and $\widetilde{S}\cong \End(\widetilde{C}_{\widetilde{R}})^{\Hop}$.
\end{itemize}
\end{example}

\section{Main Result}

To prove Theorem \ref{1.1}, we require eight lemmas. We begin with the following result, which is essential for its proof.

\begin{lemma}\label{3.1} Let $C_R$ be a cotilting $R$-module with $S=\End(C_R)$. Then there exists a duality of categories:
\begin{displaymath}
\xymatrix{^\perp(C_R)\medcap \rfmod R  \ar@<0.7ex>[rrr]^-{\Hom_R\left(-,C\right)} &
{} & {} & ^\perp(_SC)\medcap \lfmod S.  \ar@<0.7ex>[lll]^-{\Hom_S\left(-,C\right)}}
\end{displaymath}
Furthermore, for any $X,Y\in ~^\perp(C_R)\medcap \rfmod R$ and every $i\geq 0$, we have  $$\Ext^i_{R}(X,Y)\cong\Ext^i_{S}(\Hom_{R}(Y,C),\Hom_{R}(X,C)).$$
\end{lemma}

\begin{proof} It follows by  \cite[Theorem 4.2 and Propositions 5.6 and  4.5]{W}.
\end{proof}

\begin{definition}\label{3.2}  An $R$-module $C_R$ is called {\it product-complete} if $\Prod(C)\subseteq \Add(C)$.
\end{definition}

As mentioned in \cite[Remark 3.5]{MR}, an $R$-module $C_R$ is product-complete if and only if $\Prod(C)=\Add(C)$.
For an $R$-module $C_R$ where the ring $S=\End(C_R)$ is left Noetherian, by \cite[Corollary 4.4]{KS}, $C$ is
product-complete if and only if $_SC$ has finite length.

\begin{lemma}\label{3.3} Let $C_R$ be a self-orthogonal $R$-module and $n\in \mathbb{N}_0$.
\begin{itemize}
\item[(i)] $\widecheck{\add(C)}_n\medcap C^{\perp}=\add(C)=\widehat{\add(C)}_n\medcap ~^{\perp}C$.
\item[(ii)] Moreover, if $C_R$ is product-complete, then $$\widecheck{\Add(C)}_n\medcap C^{\perp}=\Add(C)=
\widehat{\Add(C)}_n\medcap ~^{\perp}C.$$
\end{itemize}
\end{lemma}

\begin{proof} (i) As the proofs of these equalities are similar, we only prove the equality on the left side.
Clearly, $$\add(C)\subseteq \widecheck{\add(C)}_n\medcap C^{\perp}.$$ Assume that $N\in\widecheck{\add(C)}_n\medcap
C^{\perp}$. We will show that $N\in \add(C)$ by induction on $n$. The claim is obvious for $n=0$. If $n=1$, then we have
an exact sequence $$0\rt N\rt C^0\rt C^1 \rt0,$$ where $C^0$ and $C^1$ are in $\add(C)$. Since $N\in C^{\perp}$ and
$C^1\in \add(C)$, this short exact sequence splits, and so $N\in \add(C)$.

Now, assume that $n>1$ and the conclusion holds for $n-1$. We have an exact sequence $$0\rt N\st{f^{-1}}\longrightarrow
C^0\st{f^0}\longrightarrow C^1\rt \cdots \st{f^{n-1}}\longrightarrow C^n \rt 0,$$  with $C^i\in \add(C)$ for all $i=0,
1, \ldots,n$. Put $L=\im(f^0)$. Since $N\in C^{\perp}$, it follows that $L\in C^{\perp}$.  So, $L\in
\widecheck{\add(C)}_{n-1}\medcap C^{\perp}.$ Hence, by the induction hypothesis, $L\in \add(C)$. Consequently,  $N\in
\add(C)$ by the case $n=1$.

(ii) Since $C_R$ is product-complete, we have $\Prod(C)=\Add(C)$. Using this fact together with the self-orthogonality
of $C_R$, it follows readily that $\Ext_R^i(X,Y)=0$ for all $X$ and $Y$ in $\Add(C)$ and every $i\geq 1$. We now proceed
exactly as in part (i).
\end{proof}

\begin{lemma}\label{3.3a} Let $C_R$ be a cotilting $R$-module, $M_R$ a finitely generated $R$-module and $n\in
\mathbb{N}_0$. If $M\in ~^{\perp}C$ and $\id(M_R)\leqslant n$, then $M\in \widecheck{\add(C)}_n$.
\end{lemma}

\begin{proof} By \cite[Propositions 5.6]{W}, there exists an exact sequence $$0\longrightarrow M\st{f^{-1}}\longrightarrow
C^0\st{f^0}\longrightarrow C^1 \st{f^1}\longrightarrow C^2\longrightarrow \cdots $$  where $C^i\in \add(C)$ and $\im(f^i)
\in~^{\perp}C$ for
all $i\geq0$. For simplicity, we put $K_0=M$ and $K_{i+1}=\im(f^i)$ for each $i\geq 0$. It is fairly easy to verify
that $C^i\in (K_{n+1})^{\perp}$ for every $i=0, \dots, n-1$. So, by Lemma \ref{2.1}(i), we have the following isomorphism $$\Ext^1_R(K_{n+1},K_n)\cong \Ext^{n+1}_R(K_{n+1},M)=0.$$ Hence, the short exact sequence $$0\rt K_{n}\rt C^{n}\rt K_{n+1}
\rt 0$$ splits, and so $K_n \in \add(C)$. Consequently, $M\in \widecheck{\add(C)}_n$.
\end{proof}

\begin{lemma}\label{3.4} Let $C_R$ be a cotilting $R$-module with $\id(C_R)=n$. Then every injective right $R$-module
belongs to $\widehat{\Add(C)}_n$.
\end{lemma}

\begin{proof} Let $S=\End(C_R)$, and let $I_R$ be an injective $R$-module. By \cite[Theorem 2.7]{Hu2}, we have $n=\id(_SC)$.
In view of \cite[Theorem 3.2.13]{EJ}, it is straightforward to see that $$\fd(\Hom_R(C,I)_S)\leq\id(_SC)=n.$$ Thus, we have
an exact sequence
\begin{equation}
0\rt F_{n+1}\rt S^{(J{n})}\rt S^{(J{n-1})}\rt S^{(J_{n-2})}\rt \cdots \rt S^{(J_0)}\rt \Hom_R(C,I)\rt 0, \label{8}
\end{equation}
where $F_{n+1}$ is a flat right $S$-module. As the $R$-module $I_R$ is injective, by the Hom-evaluation isomorphism, one
has $$\Tor_{i\geq 1}^S(\Hom_R(C,I),C)\cong \Hom_R(\Ext_S^{i\geq 1}(C,C),I)=0$$ and $$\Hom_R(C,I)\otimes_SC\cong I.$$ So,
applying the functor $-\otimes_SC$ to \eqref{8} yields the following exact sequence
$$0\longrightarrow F_{n+1}\otimes_SC\longrightarrow C^{(J_{n})}\st{d_{n}}\longrightarrow C^{(J_{n-1})}\st{d_{n-1}}
\longrightarrow C^{(J_{n-2})}\longrightarrow \cdots \longrightarrow C^{(J_0)}\st{d_0}\longrightarrow I\longrightarrow 0.$$
Set $K_{0}=I$ and $K_{j+1}=\Ker ~d_j$ for every $j=0,..., n$. By \cite[Theorem 3.2.15]{EJ}, we conclude that
\begin{equation}
\Ext_R^i(C^{(J)},K_{n+1})\cong \Ext_R^i(C,K_{n+1})^{J}\cong (F_{n+1}\otimes_S\Ext_R^i(C,C))^{J}=0,
\end{equation}
for every set $J$ and all $i\geq 1$.

By \cite[Theorem 3.2.15]{EJ}, we have
\begin{equation}
\id((K_{n+1})_R)=\id((F_{n+1}\otimes_SC)_R)\leq \id(C_R)=n.
\label{9}
\end{equation}
Hence, by applying Lemma \ref{2.1}(ii) to the exact sequence $$0\longrightarrow K_n\longrightarrow  C^{(J_{n-1})}\st{d_{n-1}}
\longrightarrow C^{(J_{n-2})}\longrightarrow \cdots \longrightarrow C^{(J_0)}\st{d_0}\longrightarrow I\longrightarrow 0$$ and
using \eqref{9}, we obtain that $$\Ext_R^1(K_{n}, K_{n+1})\cong \Ext_R^{n+1}(I, K_{n+1})=0.$$
Thus, the short exact sequence $$0\rt K_{n+1}\rt C^{(J_n)}\rt K_{n}\rt 0$$ splits, and so $K_{n}\in \Add(C)$.
\end{proof}

\begin{lemma}\label{3.5} Let $C_R$ be a finitely presented self-orthogonal $R$-module. Then, for any
non-negative integer $n$, the classe $\widehat{{\Add}(C)}_{n}$ is closed under extensions.
\end{lemma}

\begin{proof} Assume that $0\rt X\rt Y\rt Z\rt 0$ is an exact sequence with $X$ and $Z$ in $\widehat{{\Add}(C)}_{n}$.
We claim  by induction on $n$ that $Y$ in $\widehat{{\Add}(C)}_{n}$. Let $n=0$. Since $C_R$ is a finitely presented
self-orthogonal module and also $X$ and $Z$ are in $\widehat{{\Add}(C)}_0=\Add(C)$, it follows that $\Ext^1_R(Z,X)=0$.
Hence, $Y\cong X\oplus Z$, and so $Y\in \widehat{{\Add}(C)}_0$.

Next, let $n>0$ and assume that the claim holds for $n-1$. Since $X$ and $Z$ are in $\widehat{{\Add}(C)}_{n}$, we have exact
sequences $$0\rt K \rt C_0 \rt X\rt 0$$ and $$0\rt L\rt C_0'\rt Z\rt 0$$ with $C_0$ and $C_0'$ in $\Add(C)$ and $K$ and $L$
in $\widehat{{\Add}(C)}_{n-1}$. Consider the following pull-back diagram:

\[\xymatrix@C-0.5pc@R-.8pc{& &0\ar[d]&0\ar[d] &  &\\& &L\ar[d]\ar[r]^{\text{id}_L} & L\ar[d]& \\
0\ar[r]  & X \ar[r]\ar[d]^{\text{id}_X}  & U \ar[d] \ar[r] & C_0' \ar[d] \ar[r]&0&\\ 0\ar[r]& X\ar[r]  & Y\ar[r]
\ar[d] &Z  \ar[r]\ar[d]& 0. &\\ & & 0&0 &\\}\]
Since $X\in \widehat{{\Add}(C)}_{n}$, and $C$ is finitely presented, we can see that $X\in C^\perp$. As $C_0'\in \Add(C)$,
we have $\Ext^1_R(C_0',X)=0$. So, $U\cong X\oplus C_0'$. We then have the following pull-back diagram:
\[\xymatrix@C-0.5pc@R-.8pc{&0\ar[d]&0\ar[d]& &  &\\
&K\ar[d]\ar[r]^{\text{id}_K}&K\ar[d]&& \\
0\ar[r]  & V \ar[r]\ar[d]  & C_0\oplus C_0'\ar[d] \ar[r] & Y \ar[d]^{\text{id}_Y} \ar[r]&0&\\
0\ar[r]& L\ar[r]\ar[d]   & X\oplus C_0'\ar[r] \ar[d] &Y  \ar[r]& 0. &\\
&0 & 0& &}
\]
Since $K$ and $L$ are in $\widehat{{\Add}(C)}_{n-1}$, by the induction hypothesis, we get $V\in \widehat{{\Add}(C)}_{n-1}$,
and so $Y\in \widehat{{\Add}(C)}_{n}$.
\end{proof}

The subsequent result bears analogy to \cite[Lemma 2.2]{WX}, and the proof follows a similar approach.

\begin{lemma}\label{3.6} Let $C_R$ be a finitely presented self-orthogonal $R$-module, and $m, n\in \mathbb{N}$.
Let $M_R$ be an $R$-module such that there is an exact sequence $$0\rt M \rt N_1 \rt \cdots \rt N_m
\rt Z \rt 0$$ for some $R$-module $Z_R$ and $N_i \in \widehat{\Add(C)}_n $. Then there is an exact sequence
$$0\rt U \rt V \rt M \rt 0$$ for some $U\in \widehat{\Add(C)}_{n-1}$, and some $R$-module $V_R$ such that there
is an exact sequence $$0\rt V \rt C_1 \rt \cdots \rt C_m \rt Z \rt 0$$ with each $C_i \in \Add(C)$.
\end{lemma}

\begin{proof}  We prove the statement by induction on $m$. In case $m=1$, we have an exact sequence $0\rt M \rt N_1 \rt Z
\rt 0$ with  $N_1\in \widehat{\Add(C)}_n$. Since $N_1 \in \widehat{\Add(C)}_n$, there is an exact sequence $0\rt U
\rt C_1 \rt N_1 \rt 0$ with $C_1\in \Add(C)$ and $U\in \widehat{\Add(C)}_{n-1}$. Then, we can construct the following
pull-back diagram:
\[\xymatrix@C-0.5pc@R-.8pc{&0\ar[d]&0\ar[d]& &  &\\
& U\ar[d]\ar[r]^{\text{id}_{U}}&U\ar[d] && \\
0\ar[r]  & V \ar[r]\ar[d]  & C_1 \ar[d] \ar[r] & Z \ar[d]^{\text{id}_Z} \ar[r]&0&\\
0\ar[r]& M\ar[r]\ar[d]   & N_1\ar[r] \ar[d] &Z  \ar[r]& 0. &\\ &0 & 0& &}
\]
Clearly, the left column is just the desired exact sequence.

Next, assume that $m>1$ and the conclusion holds for $m-1$. Let $M'$ denote the right $R$-module $\Coker(M\rt N_1)$.
Then, by the induction assumption, there is an exact sequence $0\rt U' \rt V' \rt  M'\rt 0$ for some
$U'\in \widehat{\Add(C)}_{n-1}$ and some $V'$ such that there is an exact sequence $$0\rt V' \rt C'_1 \rt \cdots
\rt C'_{m-1} \rt Z\rt 0$$ with each $C'_i\in \Add(C)$. Now, we can construct the following pull-back diagram:
\[\xymatrix@C-0.5pc@R-.8pc{& &0\ar[d]&0\ar[d] &  &\\& &M\ar[d]\ar[r]^{\text{id}_M} &M\ar[d]& \\
0\ar[r] & U' \ar[r]\ar[d]^{\text{id}_{U'}} & Y \ar[d] \ar[r] & N_1 \ar[d] \ar[r]&0&\\
0\ar[r]& U'\ar[r]  & V'\ar[r] \ar[d] &M'  \ar[r]\ar[d]& 0. &\\ & & 0&0 &}
\]
Since $N_1$ and $U'$ are in $\widehat{\Add(C)}_{n}$, Lemma \ref{3.5} yields that $Y\in \widehat{\Add(C)}_{n}$. So,
there is an exact sequence  $0\rt U\rt X\rt  Y\rt 0$ with $X\in \Add(C)$ and $U\in \widehat{\Add(C)}_{n-1}$.
We then have the following pull-back diagram:
\[\xymatrix@C-0.5pc@R-.8pc{&&0\ar[d]&0 \ar[d] &  &&\\
0\ar[r]& U\ar[r] \ar[d]^{\text{id}_{U}} & V \ar[d]\ar[r] & M  \ar[d]\ar[r]&0& \\
0\ar[r]  & U \ar[r]  & X\ar[d] \ar[r] & Y\ar[d] \ar[r]&0&\\
&   & V'\ar[r]^{\text{id}_{V'}} \ar[d] &V'  \ar[d]&\\ & & 0&0 &}
\]
It is easy to see that the top row is just the desired sequence.
\end{proof}

\begin{corollary}\label{3.7}  Let $n\in \mathbb{N}$ and $C_R$ be a cotilting $R$-module with $\id(C_R)=n$. For any
$R$-module $M_R$, there exists an exact sequence $0\rt U\rt V \rt M\rt0$ such that $U\in
\widehat{{\Add}(C)}_{n-1}$ and $V$ fits into an exact sequence $$0\rt V \rt C_{0} \rt C_1\rt \cdots \rt C_{n-1}
\rt \Cosyz nM \rt0,$$ where each $C_i\in \Add(C)$.
\end{corollary}

\begin{proof} We consider an exact sequence $$0\rt M \rt E^0\rt\cdots\rt E^{n-1}\rt\Cosyz nM\rt0 $$ in which
each $E^i$ is an injective right $R$-module. Lemma \ref{3.4} yields that $E^i \in \widehat{\Add(C)}_n$. Now, the
conclusion follows by Lemma \ref{3.6}.
\end{proof}

\begin{lemma}\label{3.8} Let $C_R$ be a cotilting $R$-module with $\id(C_R)=n$, and $M_R$ an $R$-module.  If $M_R\in
C^{\perp}$, then $\Cosyz nM\in I^{\perp}$ for every injective $R$-module $I_R$.
\end{lemma}

\begin{proof} Let $I_R$ be an injective $R$-module. By Lemma \ref{2.1}(i), we have $$\Ext^{i}_R(I,\Cosyz nM)\cong
\Ext^{i+n}_R(I,M)$$ for all $i\geq1$.  By Lemma \ref{3.4}, we can consider an exact sequence $$0\rt C_n\rt\cdots\rt
C_1\rt C_0 \rt I\rt 0$$ with $C_i\in \Add(C)$.  It is easy to see that $\Add(C)\subseteq ~^\perp M$. So, by Lemma
\ref{2.1}(ii), we obtain $$\Ext^{i+n}_R(I,M)\cong\Ext^i_R(C_n,M)=0$$ for all $i\geq1$.   Hence,
$\Ext^{i}_R(I,\Cosyz nM)=0$  for all $i\geq1$.
\end{proof}

\begin{lemma}\label{3.9}  Let $C_R$ be a cotilting $R$-module with $\id(C_R)=n$, and let $S=\End(C_R)$. Assume that
$_SC$ is Artinian. Then, for any self-orthogonal $R$-module $M_R$ in $C^{\perp}\medcap ~^{\perp}C$, the right $R$-module
$\Cosyz nM$ is also self-orthogonal.
\end{lemma}

\begin{proof} Since $\Cosyz 0M=M$, we may and do assume that $n\geq 1$. By Lemma \ref{2.1}(i), we have
\begin{equation}
\Ext^i_{R}(\Cosyz nM,\Cosyz nM)\cong\Ext^{i+n}_{R}(\Cosyz nM,M)
\label{01}
\end{equation}
for all $i\geq 1$.

According to Corollary \ref{3.7}(a), there exists an exact sequence $0\rt U\rt V \rt M\rt 0~~~~~~     (\dagger)$ such
that $U\in \widehat{\Add(C)}_{n-1}$ and $V$ satisfies an exact sequence $$0\rt V \rt C_{0} \rt C_1\rt \cdots
\rt C_{n-1} \rt \Cosyz nM \rt0~~~~~~~~(\ddagger)$$ with $C_i\in \Add(C)$. Given that $M_R\in C^{\perp}$,
it follows that $\Add(C)\subseteq ~^\perp M$.  Applying Lemma \ref{2.1}(ii) to the exact sequence
$(\ddagger)$ give us
\begin{equation}
\Ext_R^{i+n}(\Cosyz nM,M)\cong \Ext_R^i(V,M)
\label{11}
\end{equation}
for all $i\geq 1$.

Since $M$ is self-orthogonal, applying the functor $\Hom_R(-,M)$ to $(\dagger)$ leads to the following isomorphism
\begin{equation}
\Ext^{i}_R(V,M)\cong \Ext^{i}_R(U,M)
\label{12}
\end{equation}
for all $i\geq 1$.

Because $C_R$ is self-orthogonal, it follows that $\Add(C)\subseteq ~^\perp C$. Hence, applying Lemma \ref{2.1}(ii)
to the exact sequence $(\ddagger)$, and noting that $\id(C_R)=n$, we deduce that $$\Ext_R^i(V,C)\cong
\Ext_R^{i+n}(\Cosyz nM,C)=0$$ for all $i\geq 1$. Thus, $V\in~^{\perp}C$. Since $M \in ~^{\perp}C$, the exact sequence
$(\dagger)$ yields that $U\in ~^{\perp}C$. On the other hand,  $U\in \widehat{\Add(C)}_{n-1}$. Since $_SC$ is
Artinian, by \cite[Corollary 4.4]{KS}, the $R$-module $C_R$ is product-complete. Hence, by Lemma \ref{3.3}(ii),
$U\in \Add(C)$. Consequently, $\Ext^i_R(U,M)=0$ for all $i\geq 1$. Therefore, by \eqref{01}, \eqref{11} and
\eqref{12}, we conclude that $\Ext^i_{R}(\Cosyz nM,\Cosyz nM)=0$ for all $i\geq 1$.
\end{proof}

We recall the following definition from \cite{J}. See also \cite{S}.

\begin{definition}\label{3.10} A ring $R$ is said to be {\it right co-Noetherian} if the injective hull
of every simple right $R$-module is Artinian.
\end{definition}

There is a rich source of right co-Noetherian rings as illustrated by the following example:

\begin{example}\label{3.11}
\begin{itemize}
\item[(i)] Matlis \cite{Mat} proved that the injective hull of every simple module over a commutative Noetherian ring is
Artinian, thus every commutative Noetherian ring is right co-Noetherian.
\item[(ii)] By \cite[Corollary 2.3]{Hi}, any Noetherian algebra over a commutative Noetherian ring is a right co-Noetherian
ring. In particular, the group ring $\mathbb{Z}[G]$ of a finite group $G$ is right co-Noetherian.
\item[(iii)] A ring $R$ is called right V-ring if every simple right $R$-module is injective. Clearly, every right V-ring is
right co-Noetherian.
\item[(iv)] Kaplansky's well-known criterion for commutative von Neumann regular rings implies that a commutative ring is a
V-ring if and only if it is von Neumann regular. Consequently, any infinite product of fields is a commutative right
co-Noetherian ring that is not Noetherian.
\item[(v)] By \cite[Theorem 2.5]{MV}, the property of being a right V-ring is Morita invariant. Hence, any matrix ring over
a right V-ring is again a right V-ring, yielding numerous examples of non-commutative right co-Noetherian rings.
\item[(vi)] A right Noetherian ring is not necessarily right co-Noetherian. Indeed, there exists a left and right Artinian
ring that is not right co-Noetherian; see \cite[Exercise 24.9, p. 286]{AF}.
\end{itemize}
\end{example}

\begin{theorem}\label{3.12} Let $R$ be a right Artinian ring, let $C_R$ be a cotilting $R$-module, and set $S=\End(C_R)$.
\begin{itemize}
\item[(1)] If $S$ satisfies (ARC), then $R^{\Hop}$ satisfies (DARC).
\item[(2)]  The converse holds under the additional assumptions that $R$ is right co-Noetherian and $S$ is left Artinian.
\end{itemize}
\end{theorem}

\begin{proof}
(1) Assume that $S$ satisfies (ARC). We prove that $R^{\op}$ satisfies (DARC). Let $M_R$ be
an Artinian $R$-module such that $\Ext_R^i(M \oplus I,M)=0$ for every injective $R$-module $I_R$ and all $i>0$. We show
that $M_R$ is injective.

Because $R$ is right Artinian and $M_R$ is Artinian, \cite[Corollary 2.3.24]{EJ} implies that $M_R$ is finitely generated.

Let $\id(_SC)=n$. By \cite[Theorem 2.7]{Hu2}, we have $n=\id(C_R)$. By \cite[Theorem 1.2]{Hu1}, there exists an exact
sequence
\begin{equation}
0 \rt U \rt V \rt M \rt 0  \label{111123}
\end{equation}
with $V \in ~^{\perp}C \bigcap \text{mod-}R$ and $U \in \widehat{\add(C)}_{n-1}$.

Since $\id(C_R)<\infty$ and $U \in \widehat{\add(C)}_{n-1}$, we have $\id(U_R)<\infty$. From $M \in I^{\perp}$ for
all injective $I_R$, it follows that $M \in X^{\perp}$ for every $R$-module $X_R$ of finite injective dimension.
Hence, $M \in C^{\perp} \bigcap \text{mod-}R$ and $M \in U^{\perp} \bigcap \text{mod-}R$. Since $U \in
\widehat{\add(C)}_{n-1}$, we also have $U \in C^{\perp} \bigcap \text{mod-}R$, so $V \in C^{\perp} \bigcap
\text{mod-}R$ by \eqref{111123}.

Applying $\Hom(-,M)$ to \eqref{111123} yields $$\cdots \rt \Ext^i_{R}(M,M) \rt \Ext^i_{R}(V,M) \rt \Ext^i_{R}(U,M)
\rt \cdots.$$ As $M_R$ is self-orthogonal, we get $\Ext^i_{R}(V,M)=0$ for all $i>0$. Since $V \in ~^{\perp}C
\bigcap \text{mod-}R$ and $U \in \widehat{\add(C)}_{n-1}$, we also have $\Ext^i_{R}(V,U)=0$ for all $i>0$.
Applying $\Hom_R(V,-)$ to \eqref{111123} gives $\Ext^i_{R}(V,V)=0$ for all $i>0$, hence $V \in ~(V \oplus C)^{\perp}
\bigcap ~^{\perp}C \bigcap \text{mod-}R$.

Set $N=\Hom_R(V,C)$. Lemma \ref{3.1} gives
\[\begin{array}{lllll}
 \Ext^i_{S}(N,N \oplus S) & = \Ext^i_{S}(\Hom_{R}(V,C), \Hom_{R}(V,C) \oplus \Hom_{R}(C,C)) \\
 & \cong \Ext^i_{R}(V \oplus C, V) \\
 & = 0
\end{array}\]
for all $i>0$. By assumption, $N$ is finitely generated projective over $S$, and applying Lemma \ref{3.1} again
gives $V \in \add(C)$. Then, from \eqref{111123}, we have $M \in \widehat{\add(C)}_{n}$, so $\id(M_R) < \infty$.
Consider an exact sequence $0\rt M\rt I\rt X\rt 0$ of right $R$-modules with $I$ injective and $\id(X_R)<\infty$.
Since $\Ext_R^1(X,M)=0$, applying the functor $\Hom_R(X,-)$ to this sequence shows that it splits, and consequently
$M_R$ is injective. Therefore, $R^{\op}$ satisfies (DARC).

(2) Assume that $R^{\Hop}$ satisfies (DARC).  We aim to show that $S$ satisfies (ARC).
To this end, let $N\in \lfmod S$ be such that $\Ext^{i\geq 1}_{S}(N,N\oplus S)=0$. As the ring $S$ is left
Noetherian and $_SN$ is finitely generated, we may choose syzygies of $_SN$ to be finitely generated. By Lemma
\ref{2.1}, for each $i\geq 1$, we have
\[\begin{array}{lllll}
\Ext^{i}_{S}(\Omega^n_{S}N,\Omega^n_{S}N)& \cong \Ext^{i+n}_{S}(N,\Omega^n_{S}N)\\
& \cong \Ext^{i}_{S}(N,N)\\
& =0,
\end{array}
\]
and $$\Ext^{i}_{S}(\Omega^n_{S}N,S)\cong \Ext^{i+n}_{S}(N,S)=0.$$ Hence, $\Ext^{i\geq 1}_{S}(\Omega^n_{S}N,
\Omega^n_{S}N\oplus S)=0.$ On the other hand, we know that $\id(_SC)=n$, which
implies $\Omega^n_{S}N\in ~^\perp(_SC)\medcap \lfmod S$. Thus, by Lemma \ref{3.1}, we can express $\Omega^n_{S}N\cong
\Hom_{R}(M,C)$ for some $M\in ~^{\perp}(C_R)\medcap \rfmod R$. Since $\Hom_{R}(C,C)=S$ and
$M\oplus C \in ~^{\perp}(C_R)\medcap \rfmod R$,  by Lemma \ref{3.1}, we have the following isomorphisms for every $i\geq 1$:
\begin{equation}
\begin{array}{lllll}
\Ext^{i}_{R}(M\oplus C,M)& \cong \Ext^i_{S}(\Hom_{R}(M,C),\Hom_{R}(M,C)\oplus \Hom_{R}(C,C))\\ & \cong \Ext^{i}_{S}(\Omega^n_{S}N,\Omega^n_{S}N\oplus S)\\
& =0.
\end{array}
\label{111}
\end{equation}
In particular, we see that $M_R$ is self-orthogonal and $M\in C^{\perp}\medcap ^{\perp}C$. From Lemmas \ref{3.8} and
\ref{3.9}, we obtain that
\begin{equation}
\Ext^{i}_R(\Cosyz nM\oplus I,\Cosyz nM)=0 \label{11112}
\end{equation}
for every injective $R$-module $I_R$ and for all $i\geq 1$.

As the ring $R$ is right Artinian, the finitely generated $R$-module $M_R$ is Artinian. Since the $R$-module $M_R$ is
Artinian and $R$ is right co-Noetherian, by applying \cite[Theorem 3.21]{SV}, we can choose $\Cosyz nM$ to be an Artinian
right $R$-module. As $R^{\Hop}$ satisfies (DARC), by \eqref{11112}, we induce that $\Cosyz nM$ is injective, leading
to $\id(M_R)\leq n$. So, we obtain from Lemma \ref{3.3a} that $M\in \widecheck{\add(C)}_n$. As $M\in (C_R)^{\perp}$,
Lemma \ref{3.3}(i) implies that $M\in \add(C)$. So, $\Omega^n_{S}N\cong \Hom_{R}(M,C)$ is a projective left $S$-module.
Consequently, $\pd(_SN)<\infty$. Since $N\in ~^{\perp}{S}$ too, it is easy to see that $_SN$ is a projective $S$-module.
Therefore, $S$ satisfies (ARC).
\end{proof}

\begin{remark}\label{3.12a}
It is well known that (ARC) holds for complete intersection local rings. Moreover, by \cite[Theorem 3]{Ar} and 
\cite[Theorem 4.5(1)]{CT}, the conjecture holds for all Gorenstein local rings if and
only if it holds for every Artinian Gorenstein local ring. Consequently, among commutative Noetherian rings,
Artinian Gorenstein local rings constitute a natural and fundamental class for studying (ARC). Since an Artin 
algebra $R$ satisfies (ARC) if and only if the ring $R^{\Hop}$ satisfies (DARC), it follows that, over an Artinian 
Gorenstein local ring, (ARC) and (DARC) are equivalent.

Although Theorem \ref{3.12} does not provide new instances in which (ARC) is confirmed, it shows that, under certain 
circumstances, the conjecture is equivalent to (DARC).
\end{remark}

We next give an example that provides a broad class of rings satisfying (ARC) by Theorem \ref{3.12}.
To present it, we first recall the definition of quivers and then state a lemma.

Recall that a quiver $\mathcal{Q}$ is a directed graph $(V,E)$, where $V$ and $E$ are respectively the sets of vertices and
arrows of $\mathcal{Q}$, endowed with a pair of functions $s$ and $t$ which assign to any arrow $a$ of $\mathcal{Q}$ its
source and target vertices. If $R$ is a ring and $\mathcal{Q}$ is a quiver, a representation $\mathcal{X}$ of $\mathcal{Q}$
by $R$-modules and $R$-homomorphisms is obtained by associating to each vertex $v$ an $R$-module $\mathcal{X}_v$ and to each
arrow $a : v \rt w$ an $R$-homomorphism $\mathcal{X}_a : \mathcal{X}_v \rt \mathcal{X}_w$. If $\mathcal{X}$ and $\mathcal{Y}$
are two representations of $\mathcal{Q}$, then a morphism $f : \mathcal{X}\rt \mathcal{Y}$ is determined by a family
$\lbrace f_v : \mathcal{X}_v \rt \mathcal{Y}_v \rbrace_{v\in V}$ of $R$-homomorphisms such that for any arrow $a : v \rt w$,
the commutativity condition $\mathcal{Y}_a f_v = f_w \mathcal{X}_a$ holds. For a given quiver $\mathcal{Q}$, the representations
of $\mathcal{Q}$ and the morphisms between them form a category, which is denoted by $\Rep(\mathcal{Q},R)$.

Let $\mathcal{Q}$ be a quiver. Let $R\mathcal{Q}$ be the free $R$-module with basis the set of all paths in $\mathcal{Q}$,
endowed with the following multiplication: if $p$ and $p^{'}$ are paths, let $pp^{'}$ be the concatenation of $p$ and $p^{'}$
provided the tail of $p$ is the head of $p^{'}$, and the zero vector otherwise, and extend this multiplication bilinearly to
$R\mathcal{Q}$. It is known that if $\mathcal{Q}$ is finite, that is, if $V$ and $E$ are both finite, then the category
$\Rep(\mathcal{Q},R)$ is equivalent to the category $\text{Mod-}R\mathcal{Q}$. Note that this equivalence restricts to an
equivalence of categories $\rep(\mathcal{Q},R)$ and $\text{mod-}R\mathcal{Q}$. Here, $\rep(\mathcal{Q},R)$ denotes the
subcategory of $\Rep(\mathcal{Q},R)$ consisting of all finitely generated representations. The proof of this assertion is
similar to the case where $R$ is a field; see \cite[Chapter III, Theorem 1.6]{ASS}.

\begin{lemma}\label{3.13}
Let $\mathcal{Q}\colon 1 \rt 2$ be a quiver. Then:
\begin{itemize}
\item[(i)] If $R$ is a right Artinian ring, then $R\mathcal{Q}$ is also a right Artinian ring.
\item[(ii)] If $R$ is right co-Noetherian, then $R\mathcal{Q}$ is also a right co-Noetherian ring.
\item[(iii)] $R\mathcal{Q}\cong \left(\begin{matrix}
{R}& {R}\\
{0}&{R}
\end{matrix}
\right)$.
\item[(iv)] $(R\mathcal{Q})^{\Hop}\cong R^{\Hop}\mathcal{Q}^{\Hop}$.
\end{itemize}
\end{lemma}

\begin{proof} (i) Since $R$ is a right Artinian ring and $R\mathcal{Q}$ is a finitely generated free $R$-module, it
follows immediately that the ring $R\mathcal{Q}$ is also right Artinian.

(ii) Let $M$ be an arbitrary simple $R\mathcal{Q}$-module. We need to show that the injective hull of $M$ is Artinian.
The representation of $M$ in $\Rep(\mathcal{Q},R)$ is either $r_1=(S \rt 0)$ or $r_2=(0 \rt S)$, where $S$ is a simple
$R$-module. On the other hand, by \cite[Theorem 3.1]{AHK}, the injective hulls of $r_1$ and $r_2$ are given by $E(r_1)
=(E(S) \rt 0)$ and $E(r_2)=(E(S) \rt E(S))$, respectively. Since $R$ is a right co-Noetherian ring, $E(S)$ is an Artinian
$R$-module. Consequently, the injective hull of $M$ is an Artinian $R\mathcal{Q}$-module. Therefore, the ring $R\mathcal{Q}$
is right co-Noetherian.

The proofs of assertions (iii) and (iv) are analogous to the case where $R$ is a field; see \cite[Lemma II.1.12]{ASS} and
\cite[Exercise II.4.1]{ASS}, respectively.
\end{proof}

\begin{example}\label{3.14} Let $\Gamma$ be a right co-Noetherian ring, and let $E_\Gamma$ be a cotilting module.
Set $\Lambda=\End(E_\Gamma)$, and assume that $\Gamma$ is right Artinian and that $\Lambda$ is left Artinian. Put
$R=\left(\begin{matrix}
{\Gamma}& {\Gamma}\\
{0}&{\Gamma}
\end{matrix}
\right)$  and
${S}=\left(\begin{matrix}
{\Lambda}& {\Lambda}\\
{0}&{\Lambda}
\end{matrix}
\right)$.
Set $C=(0\rt E)\oplus (E\rt E)$. By Example \ref{2.6}(iv), $C_R$ is a cotilting $R$-module and $S\cong \End(C_R)^{\Hop}$.
By Lemma \ref{3.13}, the ring $R$ is right Artinian and right co-Noetherian. Moreover, the ring $S$ is left Artinian.
If $R^{\Hop} $ satisfies (DARC), then it follows from Theorem \ref{3.12} that $S$ satisfies (ARC) conjecture.
\end{example}

\section*{Acknowledgement} The authors thank the referee for their careful reading of the manuscript and for the
valuable and constructive comments, which have significantly improved the paper. In particular, addressing the
referee’s second comment led to a substantial improvement of the main theorem. They also thank Asghar Daneshvar
for introducing the papers \cite{J, Hi} and for his insightful comments on this work.


\end{document}